\documentclass[12pt]{amsart}
\usepackage{amsmath, amssymb, amsthm, bm, float, mathtools, enumitem, caption}
\usepackage{hyperref}

\usepackage[noabbrev,capitalize]{cleveref}
\usepackage{adjustbox}
\crefname{equation}{}{}
\usepackage{ytableau}
\usepackage{graphicx, tikz}
\usepackage[textheight=8.5in, textwidth=6.75in]{geometry}
\usepackage{subcaption}
\usepackage{microtype}

\newtheorem{theorem}{Theorem}[section]
\newtheorem{lemma}[theorem]{Lemma}
\newtheorem{corollary}[theorem]{Corollary}

\newtheorem*{conjecture*}{Conjecture}

\theoremstyle{definition}

\theoremstyle{remark}
\newtheorem*{remark}{Remark}
\newtheorem*{example}{Example}

\numberwithin{equation}{section}

\newcommand{\Z}{\mathbb Z}

\title[Hook lengths in self-conjugate partitions]{Hook lengths in self-conjugate partitions}
\date{\today}
\thanks{2020 {\it{Mathematics Subject Classification.}} {05A15, 05A17, 11P81, 11P83}}
\keywords{self-conjugate partitions, hook lengths, $t$-core, $t$-quotient}
\author{Tewodros Amdeberhan, George E. Andrews, Ken Ono \and Ajit Singh}
\address{Dept. of Mathematics, Tulane University, New Orleans, LA 70118}
\email{tamdeber@tulane.edu}
\address{Dept. of Mathematics, Penn State University, University Park, PA 16802}
\email{gea1@psu.edu}
\address{Dept. of Mathematics, University of Virginia, Charlottesville, VA 22904}
\email{ko5wk@virginia.edu}
\email{ajit18@iitg.ac.in}

\begin{document}
\begin{abstract} 
In 2010, G.-N. Han obtained the generating function for the number of size $t$ hooks among integer partitions.  Here we obtain these generating functions for self-conjugate partitions, which are particularly elegant for even $t$. If $n_t(\lambda)$ is the number of size $t$ hooks in a partition $\lambda$ and $\mathcal{SC}$ denotes the set of self-conjugate partitions, then for even $t$ we have
$$\sum_{\lambda\in \mathcal{SC}} x^{n_t(\lambda)} q^{\vert\lambda\vert}
=  (-q;q^2)_{\infty} \cdot ((1-x^2)q^{2t};q^{2t})_{\infty}^{\frac{t}{2}}.
$$
As a consequence, if $a_t^{\star}(n)$ is the number of such hooks among the self-conjugate partitions of $n,$ then for even $t$ we obtain the simple formula
$$
a_t^{\star}(n)=t\sum_{j\geq 1} q^{\star}(n-2tj),
$$
where $q^{\star}(m)$ is the number of partitions of $m$ into distinct odd parts. As a corollary, we find that $t\mid a_t^{\star}(n),$ which confirms a conjecture of Ballantine, Burson, Craig, Folsom and Wen. 
\end{abstract}

\maketitle
\section{Introduction and Statement of Results}

Integer partitions play many roles in mathematics. They are fundamental objects in combinatorics, geometry, mathematical physics, number theory,  and representation theory. In particular, Young diagrams (a.k.a. Ferrers graphs) of partitions and their hook lengths are prominent. For example, they arise in the study of class numbers of imaginary quadratic fields \cite{OnoSze}, the Seiberg-Witten theory of random partitions developed by Nekrasov and Okounkov \cite{NekOk}, Ramanujan's partition congruences \cite{GKS}, and the representation theory of the symmetric group \cite{James-Kerber}, to name a few applications. In this note, we fill a gap in the literature by counting the number of size $t$ hooks  in self-conjugate partitions. As a consequence, we  prove a recent conjecture about the congruence properties of these numbers.

To make this precise, we recall the fundamental definitions (for example, see \cite{Andrews}).
A non-increasing sequence of natural numbers $\lambda=(\lambda_1,\lambda_2,\dots,\lambda_{\ell})$ is called a \emph{partition} of $n$, denoted $\lambda\vdash n$,  where $\lambda_1$ through $\lambda_{\ell}$ are called its \emph{parts}, provided these add up to $n$. 
Let $\vert\lambda\vert:=\sum\lambda_i$ denote its \emph{size}.  Partitions are represented by \emph{Young diagrams}, where the parts are arranged in left-justified rows. The {\it conjugate} of a partition $\lambda$,
denoted by $\lambda'$, is associated to the Young diagram obtained by reflecting  the diagram for  $\lambda$ across
the main diagonal. We say that $\lambda$ is  \emph{self-conjugate} if $\lambda=\lambda'$. Here we study $\mathcal{SC},$ the set of all self-conjugate partitions.

For a given partition $\lambda$, a cell at $(i, j)$ in the Young diagram of $\lambda$ is the box in the $i^{th}$-row
from the top and the $j^{th}$-column from the left. The \emph{hook} of a cell at $(i, j)$ contains the boxes at
$\{(k, j) :\, k \geq i\} \,\cup\,  \{(i, k) :\,  k \geq j\}$. The \emph{hook length} $h$ of a box at $(i, j)$ is the number of boxes directly to the right and directly below the box, plus $1$ for the box itself. We denote the multiset of hook lengths by $\mathcal{H}(\lambda)$.

We are interested in the partition statistic
\begin{equation}
n_t(\lambda):=\#\{h\in\mathcal{H}(\lambda):\, h=t\},
\end{equation}
 the number of length $t$ hooks in a partition $\lambda.$ 
For the set of all integer partitions $\mathcal{P},$ Han (see Thm. 1.4 of \cite{Han}) obtained the bivariate generating function for $n_t(\lambda)$ as an infinite product:
\begin{align}\label{Han-Thm1.4} 
\sum_{\lambda\in\mathcal{P}}  x^{n_t(\lambda)}\, q^{\vert\lambda\vert}
=\prod_{k\geq1}\frac{(1+(x-1)q^{tk})^t}{1-q^k}=\frac{((1-x)q^t;q^t)_{\infty}^t}{(q;q)_{\infty}},
\end{align}
where we have used  the {\it $q$-Pochhammer symbol}
 \begin{equation}
 (a;q)_n:= \begin{cases} 1   \ \ \ \ \ & {\text {\rm if}}\ n=0,  \\
(1-a)(1-aq)\cdots (1-aq^{n-1}) \ \ \ \ \ &{\text {\rm if}}\ n\in \Z_+,\\
  \prod_{j=0}^{\infty}(1-aq^j) \ \ \ \ \ &{\text {\rm if}}\ n=+\infty.
  \end{cases}
 \end{equation}
Here we obtain the corresponding generating functions for the set of self-conjugate partitions $\mathcal{SC}.$

\begin{theorem} \label{vital} The following are true.

\noindent
(1) If $t$ is even, then we have 
$$
\sum_{\lambda\in \mathcal{SC}} x^{n_t(\lambda)}\, q^{\vert\lambda\vert}  
= (-q;q^2)_{\infty}\cdot ((1-x^2)q^{2t};q^{2t})_{\infty}^{\frac{t}{2}}.
$$

\noindent
(2)  If $t$ is odd, then  we have 
$$\sum_{\lambda\in \mathcal{SC}}  x^{n_t(\lambda)} \, q^{\vert\lambda\vert} =  (-q;q^2)_{\infty}\cdot H^{\star}(x;q^t)\cdot
((1-x^2)q^{2t}; q^{2t})_{\infty}^{\frac{t-1}{2}},
$$
where $H^{\star}(x;q)$ is defined in (\ref{HstarFormula}).
\end{theorem}

\begin{remark}
The $q$-series $H^{\star}(x;q)$ is derived from the generating function for the number of 1-hooks in self-conjugate partitions by the formula
\begin{equation}\label{Hstar}
H^{\star}(x;q):=\frac{1}{{(-q;q^{2})_{\infty}}}\cdot \sum_{\lambda\in \mathcal{SC}}  x^{n_1(\lambda)} q^{\vert\lambda\vert}.
\end{equation}
The formula in (\ref{HstarFormula}) uses Theorem~\ref{newconj} (2), which offers an explicit expression for the inner sum.
\end{remark}

Motivated by work of Ballantine, Burson, Craig, Folsom and Wen \cite{BBCFW} and Craig, Dawsey and Han \cite{CDH}, we study
\begin{equation}
a_t^{\star}(n):=\sum_{\substack {\lambda \in \mathcal{SC}\\ \lambda \vdash n}} n_t(\lambda),
\end{equation}
the number of size $t$ hook lengths in the self-conjugate partitions of $n$.
We obtain the generating function for $a_t^{\star}(n),$ which gives formulas in terms of $sc(m)$, the number of self-conjugate partitions of $m$ (resp. $q^{\star}(m)$, the number of partitions of $m$ into distinct odd parts).

\begin{theorem}\label{ExactFormula} The following are true.

\noindent
(1) If $t$ is even, then we have 
$$\sum_{n\geq1} a_t^{\star}(n)\, q^n=t \cdot \frac{q^{2t}\cdot (-q;q^2)_{\infty}}{1-q^{2t}}.
$$
Furthermore, we have 
$$a_t^{\star}(n)=t\sum_{j\geq 1} sc(n-2tj)=t\sum_{j\geq 1} q^{\star}(n-2tj).
$$

\noindent
(2)  If $t$ is odd, then we have 
$$
\sum_{n\geq1} a_t^{\star}(n) \, q^n=\frac{q^t(1+(t-1)q^t+tq^{2t})}{(1-q^{2t})(1+q^t)}\cdot (-q;q^2)_{\infty}.
$$
Furthermore,  we have 
\begin{displaymath}
\begin{split}
a_t^{\star}(n)&=\sum_{j\geq 1}\left((-1)^{j-1}j \cdot  sc(n-tj)+t\cdot  sc(n-2tj)\right)\\
&=\sum_{j\geq 1}\left((-1)^{j-1}j \cdot q^{\star}(n-tj)+t\cdot q^{\star}(n-2tj)\right).
\end{split}
\end{displaymath} 
\end{theorem}

\begin{example} Figure~\ref{hooks_in_SC} gives the self-conjugate partitions of 16, and Table~\ref{a_t(16)} lists the values of $a_t^{\star}(16).$

\begin{figure}[H]
		\begin{center}
	\includegraphics[height=47mm]{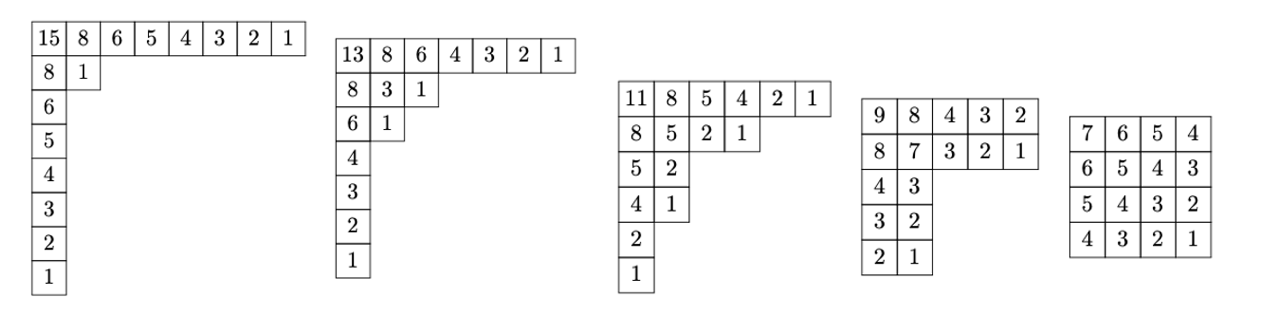}
	\end{center}
		\caption{Hook lengths of the self-conjugate partitions of $16$}
                \label{hooks_in_SC}
	\end{figure}

	\begin{table}[H]
		\begin{tabular}{|c|c|c|c|c|c|c|c|c|c|c|c|c|c|c|c|c|}
			\hline$t$ &1 & 2 & 3 & 4 & 5 & 6 & 7 & 8 & 9 & 10 & 11 & 12& 13 &14&15 & $\geq 16$ \\
			\hline$a_t^{\star}(16)$ & 14 & 14 & 12 & 12 & 8 & 6 & 2 & 8 & 1 &0 &1 &0 &1&0&1&0\\
			\hline
		\end{tabular}
		\caption{Values of $a_t^{\star}(16)$}
                      \label{a_t(16)}
	\end{table}
	
\noindent
The first few terms of the generating function for $sc(n)$ and $q^{\star}(n)$ are:
\begin{displaymath}
\begin{split}
\sum_{n=0}^{\infty} q^{\star}(n)q^n&=\sum_{n=0}^{\infty}sc(n)q^n=\prod_{n=0}^{\infty}(1+q^{2n+1})\\
&=1+q+q^3+q^4+q^5+q^6+q^7+2q^8+2q^9+2q^{10}+2q^{11}+3q^{12}+\cdots.
\end{split}
\end{displaymath}
Using these terms, we illustrate Theorem~\ref{ExactFormula} (1) with the identities:
\begin{displaymath}
\begin{split}
a_2^{\star}(16)&=2\left(q^{\star}(16-4)+q^{\star}(16-8)+q^{\star}(16-12)+q^{\star}(0)\right)=14,\\
a_4^{\star}(16)&=4\left(q^{\star}(16-8)+q^{\star}(16-16)\right)=12,\\
a_6^{\star}(16)&=6q^{\star}(16-12)=6,\\
a_8^{\star}(16)&=8q^{\star}(16-16)=8.
\end{split}
\end{displaymath}

\end{example}

Ballantine, Burson, Craig, Folsom and Wen made three conjectures \cite{BBCFW} (also see the last section of \cite{CDH}) about $a_t^{\star}(n).$ One of these conjectures concerns the divisibility properties of $a_t^{\star}(n),$ which is illustrated in the above example.

\begin{conjecture*}[BBCFW]
For all integers $n\geq0$ and $m\geq1$, we have 
$$a_{2m}^{\star}(n)\equiv 0 \pmod{2m}.$$
\end{conjecture*}

By letting $t=2m$ in Theorem~\ref{ExactFormula} (1), we immediately obtain the following result.

\begin{corollary} \label{conj} The  BBCFW Conjecture is true. 
\end{corollary}

To prove these results, we make use of the classical Littlewood bijective decomposition which associates a $t$-core partition and a $t$-quotient to every partition $\lambda \in \mathcal{P}.$  More precisely, we use P\'etr\'eolle's \cite{Pet} work on the restriction of these bijections to self-conjugate partitions. We recall some of his work in Section 2.  In Section 3, we recall some classical partition generating functions, and we obtain two formulas for the 1-hook generating function (see Theorem~\ref{newconj})
$$
\sum_{\lambda\in \mathcal{SC}}  x^{n_1(\lambda)} q^{\vert\lambda\vert}.
$$
We then combine these results with the restricted Littlewood decomposition to prove Theorems~\ref{vital} and ~\ref{ExactFormula}.

\section*{Acknowledgements}
\noindent
The second author is partially supported by Simon Foundation Grant 633284. The third author thanks the Thomas Jefferson Fund and the NSF
(DMS-2002265 and DMS-2055118). The fourth author thanks the support of a Fulbright Nehru Postdoctoral Fellowship. Moreover, the fourth author thanks his Ph.D. advisor Rupam Barman and mentors Eknath Ghate and Anupam Saikia for their support and lifelong encouragement.

\section{The refined Littlewood decomposition}

The Littlewood partition bijections \cite{Littlewood}, one for every positive integer $t$,  are fundamental combinatorial tools in the representation theory of the symmetric group. Given $t$, the bijection associates partitions of $n$ to a pair consisting of a $t$-core and a $t$-quotient. The combinatorics of these bijections underly many important generating functions (for example, see \cite{GKS, Han, James-Kerber}).
In an analogous way, the results we obtain are based on the restriction of these bijections to self-conjugate partitions.

To make this precise, suppose that $t$ is a fixed positive integer.
We say that a partition is a \emph{$t$-core} if none of its hook lengths equal $t$ (equivalently, none are multiples of $t$). Given a subset of partitions $A$, write $A_t$ for those members that are $t$-cores.  Denote the multiset of hook lengths that are multiples of $t$ by $\mathcal{H}_t(\lambda)$.

The Littlewood bijection $\phi_t$ (see \cite[Thm. 2.7.30]{James-Kerber}) associates to
every partition $\lambda\in\mathcal{P}$ a $t$-core 
$\omega\in\mathcal{P}_t$ and a \emph{$t$-quotient} $(\nu^{(0)},\nu^{ (1)},\dots,\nu^{(t-1)})\in\mathcal{P}^t$. This bijection 
$\phi_t: \mathcal{P}\rightarrow\mathcal{P}_t\times\mathcal{P}^t$ is defined by
\begin{equation}\label{phit}
\phi_t(\lambda):=(\omega;\nu^{(0)},\dots,\nu^{(t-1)}),
\end{equation}
where 
$\vert\lambda\vert=\vert\omega\vert+t\sum_{i=0}^{t-1}\vert\nu^{(i)}\vert$. 
We briefly  recall its construction.

Let $\mathcal{B}$ be the set of all bi-infinite binary sequences beginning with (resp. ending with)  $0$'s (resp.  $1$'s). We denote $c\in\mathcal{B}$ by a sequence $(c_i)_{i\in\mathbb{Z}}=\cdots c_{-2}c_{-1}c_0c_1c_2\cdots$, although such a representation is not unique (for example, simple shifts work). The \emph{canonical representation} of $c$ is indexed by choosing $c_0$ so that
$$\#\{i\leq-1: \, c_i=1\}=\#\{i\geq0: \, c_i=0\}.$$
We place a dot ``{\bf .}'' between  $c_{-1}$ and $c_0$ for notational convenience.

This set of bi-infinite sequences $\mathcal{B}$ is a convenient tool for representing  partitions.
To be precise, let $\lambda$ be a partition with Young diagram $Y.$  The boundary of $Y$ consists of a vertical line on the left, which captures the number of parts, and a horizontal line at the top, which gives the size of the largest part $\lambda_1$. We focus on the remaining piece, the right hand boundary, which generally has the shape of a ``staircase'' (a.k.a. ``rim'').  We mark each horizontal edge (resp. each vertical edge) by a $1$ (resp. by at $0$). The resulting sequence of digits obtained by traversing the rim from the bottom to the top is a binary word $u$. By appending  $0$'s (resp. $1$'s) to $u$, we obtain an element
$c_{\lambda}\in\mathcal{B}$.  Obviously, this defines a bijection $\varphi :\mathcal{P}\mapsto \mathcal{B}.$

 We refine $\varphi$ to obtain each Littlewood bijection $\phi_t.$
Given $t$, one views $c_{\lambda}=(c_i)_{i\in \Z}$ as  $t$ many interlaced subsequences dictated by the residue classes of the indices modulo $t$. Namely, one forms the subsequences $b^{(k)}:=(c_{it+k})_{i\in \Z}$ for $0\leq k \leq t-1.$ The inverse image $\varphi^{-1}(b^{(k)})$ gives a partition $\nu^{(k)},$ and the collection of these inverse images for $0\leq k\leq t-1$ defines a vector of partitions that we call the $t$-quotient of  $\lambda$. This is the second component of
$\phi_t(\lambda)$ in (\ref{phit}).

To construct the first entry of $\phi_t(\lambda),$ the $t$-core, for each $0\leq k\leq t-1,$ we replace each subword $10$ with $01$ in $b^{(k)}$ while reading from left to right. We continue until we obtain $\cdots000111\cdots,$ which we denote by $r^{(k)}$.  As above, we view the $r^{(k)}$ as subsequences of a uniquely determined $r\in \mathcal{B},$ and we let  $\omega:=\varphi^{-1}(r).$ There are two major features of this construction. It turns out that $\omega$ is always a $t$-core partition, and  that $\phi_t$ as given in (\ref{phit}) is a bijection.

In his thesis, P\'etr\'eolle studied the restriction of the $\phi_t$ to self-conjugate partitions. Here we recall one of his main results \cite[Thm. 2.1]{Pet} (see also \cite{GKS}, \cite{Wah}) which refines Littlewood's bijection to the setting of self-conjugate partitions.

\begin{theorem}\label{LD} Let $t$ be a positive integer. The Littlewood decomposition $\phi_t$ maps a self-conjugate partition $\lambda$ to $(\omega;\nu^{(0)},\dots,\nu^{(t-1)}):=(\omega;\underline{\nu})$ such that the following are true.

\noindent
\qquad (SC1) We have that $\omega\in\mathcal{SC}_t$ is a $t$-core and $(\nu^{(0)},\dots,\nu^{(t-1)})$ are partitions.

\noindent
\qquad (SC2) We have that $\nu^{(j)}=(\nu^{(t-1-j)})'$ for each $j\in\{0,1,\dots,\lfloor t/2\rfloor-1\}$.

\noindent
\qquad (SC2$\,'$) If $t$ is odd, then $\nu^{(\frac{t-1}{2})}=(\nu^{(\frac{t-1}{2})})':=\mu\in\mathcal{SC}$.
	
\noindent
\qquad (SC3) We have that
	 $$
	 \vert\lambda\vert=\begin{cases}\vert\omega\vert+2t\sum_{i=0}^{\frac{t}{2}-1}\vert\nu^{(i)}\vert \qquad \ \ \ &\text{if $t$ is even} \\                                                                  		          \vert\omega\vert+2t\sum_{i=0}^{\frac{t-3}{2}}\vert\nu^{(i)}\vert+t\vert \mu\vert \qquad  &\text{if $t$ is odd}.
                                                                 \end{cases}
$$                                                                 
                                                             
\noindent
\qquad (SC4) We have that $\mathcal{H}_t(\lambda)=t\,\mathcal{H}(\underline{\nu})$. Moreover, $\phi_t$  gives a 1-1 correspondence between these 

\qquad \qquad  multisets.
\end{theorem}

\begin{example} We illustrate Theorem~\ref{LD} with the self-conjugate partition $\lambda=(7,7,5,4,3,2,2)$ and $t=4$.  
 The partition $\lambda=(7,7,5,4,3,2,2)\vdash 30$ produces the canonical bi-infinite binary sequence $\dots 0001100101.0101100111\dots$.
Table \ref{binary} gives the ``non-trivial" portions (i.e. without the infinitely many 0s and 1s) of the binary sequences for the respective partitions in the Littlewood decomposition.  The reader should note that the entries for $\nu^{(*)}$ and $r^{(*)}$  are in fixed arithmetic progressions modulo 4, which explains the empty spaces in the middle rows of the table.  After compiling the $r^{(*)},$ we obtain the bi-infinite sequence $\omega$, which one checks gives the $4$-core partition and $4-$quotient in Figure~\ref{LDyoung}. This establishes that $\mathcal{H}_4(\lambda)=4\mathcal{H}(\underline{\nu}),$ where
$\mathcal{H}_4(\lambda)=\{{\color{orange}12,12},{\color{green}8,8},{\color{red}4,4}\}$ and $\mathcal{H}(\underline{\nu})=\{\,{}\, \}\cup\{{\color{orange}3},{\color{green}2},{\color{red}1}\}\cup \{{\color{orange}3},{\color{green}2},{\color{red}1}\} \cup \{\,{}\,\}$.
The coloring indicates the correspondence induced by $\phi_4.$

\begin{figure}[H]
\ytableausetup{centertableaux}
\begin{ytableau}
13&*(orange)12&9&7&5&3&2 \\
*(orange)12&11&*(green)8&6&*(red)4 &2 &1 \\
9&*(green)8&5&3&1  \\
7&6&3&1 \\
5&*(red)4&1 \\
3 & 2\\
2 & 1
\end{ytableau} 
\caption{$\lambda=(7,7,5,4,3,2,2)\in\mathcal{SC}$ with hook lengths inserted}
\label{SCyoung}
\end{figure}

\begin{table}[H] 
\caption{The Littlewood decomposition in binary digits}
\label{binary}
$\begin{array}{c|ccccccccccccccccccc}
  \lambda &\cdots&1&1&0& 0&1 &  0&   1 &  \text{\bf .}  & 0   & 1  &0&1&   1 &   0  & 0&  \cdots  \\ \hline
   \nu^{(0)} &\cdots& & & & 0   &   &   &     & &0   &  &   & & 1   &  &  & \cdots \\
   \nu^{(1)} &\cdots&1 & & &    &1   &   &     & &  &1  &   & &   &0  &  & \cdots \\
 \nu^{(2)} &\cdots& &1 & &    &  & 0  &     & &  &  &0   & &  &  & 0 & \cdots  \\
   \nu^{(3)} &\cdots &  & &0  &  &  &  & 1& &   &  &   &  1 & & & & \cdots \\ \hline
r^{(0)} &\cdots& & & & 0   &   &   &     & &0   &  &   & & 1   &  &  &  \cdots \\
   r^{(1)} &\cdots&0 & & &    &1   &   &     & &  &1  &   & &   &1  &  & \cdots \\
r^{(2) }&\cdots& &0 & &    &  & 0  &     & &  &  &0   & &  &  & 1 & \cdots \\
   r^{(3)} &\cdots &  & &0  &  &  &  & 1& &   &  &   &  1 & & & & \cdots \\ \hline
  \omega  &\cdots&0&0&0& 0&1 &  0&   1 &  \text{\bf .}  & 0   & 1  &0&1&   1 &   1  & 1& \cdots  \\
 \end{array}$
\end{table}

\begin{figure}[ht!]
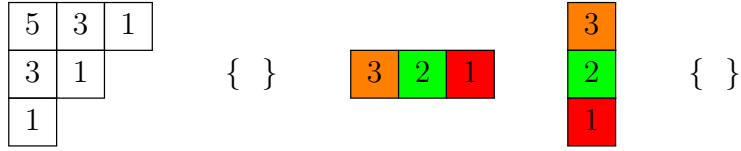

\begin{ytableau}
5&3&1 \\
3&1 \\
1
\end{ytableau}
\qquad
\{ {} \}
 \qquad
\begin{ytableau}
*(orange)3&*(green)2&*(red)1
\end{ytableau}
\qquad
\begin{ytableau}
*(orange)3 \\ *(green)2 \\ *(red)1
\end{ytableau}
\qquad
\{ {} \}
\caption{The $4$-core $\omega$ and partitions $\nu^{(0)}, \nu^{(1)} , \nu^{(2)}, \nu^{(3)}$ for $\lambda$}
\label{LDyoung}
\end{figure}
\end{example}

\section{Proofs of Theorem 1.1 and Theorem 1.2.}

Here we prove Theorems~\ref{vital} and \ref{ExactFormula}. These results follow from P\'etr\'eolle's work on Littlewood's decompositions, combined with some known partition generating functions,  and a more involved analysis in the case of $1$-hooks in self-conjugate partitions.

\subsection{Counting 1-hooks in self-conjugate partitions}

We begin by counting the number of  1-hooks in self-conjugate partitions. The next theorem offers two expressions for the relevant generating function.

\begin{theorem}\label{newconj}  The following are true.

\noindent
(1) We have that
\begin{align*}
\sum_{\lambda\in \mathcal{SC}}  x^{n_1(\lambda)} \, q^{\vert\lambda\vert} 
&=\frac{(-q;q^2)_{\infty}}{2x}\left[\left(1-\sqrt{\frac{1-x}{1+x}}\right)(-\sqrt{1-x^2};-q)_{\infty} + \left(1+\sqrt{\frac{1-x}{1+x}} \right)(\sqrt{1-x^2};-q)_{\infty}  \right]. 
\end{align*}

\vskip.05in
\noindent
(2) We have that
\begin{align*}
\sum_{\lambda\in \mathcal{SC}}  x^{n_1(\lambda)} \, q^{\vert\lambda\vert}
&= (-q;q^2)_{\infty}\cdot \left[\left(1-\frac1x\right)\sum_{n\geq0}\frac{(x^2-1)^nq^{2n^2+n}}{(q^2;q^2)_n(-q;q^2)_{n+1}}+\frac1x\,\sum_{n\geq0}\frac{(x^2-1)^nq^{2n^2-n}}{(q^2;q^2)_n(-q;q^2)_n}\right].
\end{align*}
\end{theorem}

In view of (\ref{Hstar}), both identities in Theorem~\ref{newconj} can in principle be used to define the $q$-series $H^{\star}(x;q)$ required by Theorem~\ref{vital}. For convenience, we choose the second formula, and we define
\begin{equation}\label{HstarFormula}
H^{\star}(x;q):= \left[\left(1-\frac1x\right)\sum_{n\geq0}\frac{(x^2-1)^nq^{2n^2+n}}{(q^2;q^2)_n(-q;q^2)_{n+1}}+\frac1x\,\sum_{n\geq0}\frac{(x^2-1)^nq^{2n^2-n}}{(q^2;q^2)_n(-q;q^2)_n}\right].
\end{equation}
To prove Theorem~\ref{ExactFormula} (2),  
we will differentiate $H^{\star}(x;q)$ in $x$ and let $x\rightarrow 1.$ These steps would require more care had we chosen the formula in Theorem~\ref{newconj} (1) due to the vanishing of $1-x^2.$

The proof of Theorem~\ref{newconj} requires a simple observation which leads to two $q$-series. Using Young diagrams, we have that every self-conjugate partition $\lambda$ can be divided into three (possibly empty) partitions.
The Durfee square gives one partition, and the remaining cells define a pair of identical twin partitions. The columns to the right of the Durfee square give the same partition as the rows below the square.  

There are two natural types of non-empty self-conjugate partitions. We say that $\lambda$ is of {\it Type~1} if the width of the Durfee square is strictly larger than the largest part of the twins. Otherwise, we say that $\lambda$ is of {\it Type~2}, which means that the width of the Durfee square equals the largest part of the twins. Therefore, if we let $\mathcal{SC}_{(1)}$ (resp. $\mathcal{SC}_{(2)})$ denote these subsets, and also account for the empty partition of 0, then we have that
$$
\sum_{\lambda\in \mathcal{SC}}x^{n_1(\lambda)}q^{|\lambda|}= 1+D_1(x;q)+D_{2}(x;q),
$$
where
\begin{eqnarray}
D_1(x;q):=&\sum_{\lambda\in \mathcal{SC}_{(1)}}x^{n_1(\lambda)}q^{|\lambda|}\label{D1}\\
D_2(x;q):=&\sum_{\lambda\in \mathcal{SC}_{(2)}}x^{n_1(\lambda)}q^{|\lambda|}.\label{D2}
\end{eqnarray}
The proof of Theorem~\ref{newconj} boils down to an explicit determination of these separate generating functions.

\begin{example}  Figure \ref{AjitsPictures} offers two self-conjugate partitions of 17, both with $3\times3 $ Durfee squares. The Durfee squares are colored orange, and the twin partitions are colored green. The partition on the left is of Type 1, and the other is of Type 2. The 1-hooks are labelled in both Young diagrams.

		\begin{figure}[H]
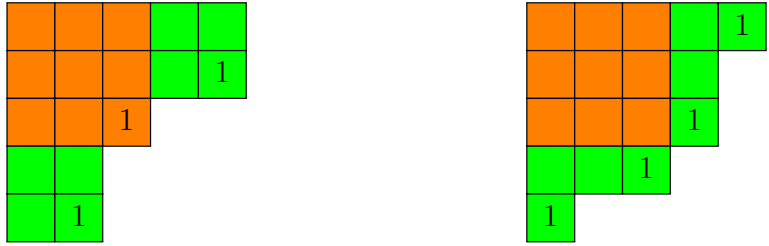

			\ytableausetup{centertableaux}
			\begin{ytableau}
				*(orange)&*(orange)&*(orange)&*(green)&*(green) \\
				*(orange)&*(orange)&*(orange)&*(green)&*(green)1 \\
				*(orange)&*(orange)&*(orange)1 \\
				*(green)&*(green)\\
				*(green)&*(green)1			
			\end{ytableau}\ \ \ \ \ \ \ \ \ \ \ \ \ \ \ \ \ \ \ \ \ \ \ \ \ \ \
			\begin{ytableau}
				*(orange)&*(orange)&*(orange)&*(green)&*(green)1 \\
				*(orange) &*(orange)&*(orange)&*(green) \\
				*(orange) &*(orange)&*(orange)&*(green)1 \\
				*(green)&*(green)&*(green)1\\
				*(green)1
			\end{ytableau} 
			\caption{Self-conjugate partitions of $17$ of Type 1 and Type 2}\label{AjitsPictures}
		\end{figure}
\end{example}

To obtain explicit expressions for these two generating functions, we require the next lemma.

\begin{lemma} \label{George_Lm} The following are true.

\noindent
 (1) If we let $F(A;q):=\sum_{n\geq0}\frac{q^{n^2}(A;q^2)_n}{(q^2;q^2)_n}$, then we have 
$$F(A;q)=\frac12 \cdot (-q;q^2)_{\infty}\left[(-\sqrt{A};-q)_{\infty}+(\sqrt{A};-q)_{\infty}\right].$$

\noindent
(2)  If we let $G(A;q):=\sum_{n\geq0}\frac{q^{n^2+2n}(A;q^2)_n}{(q^2;q^2)_n}$, then we have 
$$G(A;q)=\frac1{2\sqrt{A}} \cdot (-q;q^2)_{\infty}\left[(-\sqrt{A};-q)_{\infty}- (\sqrt{A};-q)_{\infty}\right].$$
\end{lemma}
\begin{proof} \ \newline

\noindent
(1) By applying Heine's transformation \cite[p.19, Cor.2.3]{Andrews} to $F(A;q)$, we obtain
\begin{align*}
F(A;q)&:=\lim_{y\rightarrow0}\, \setlength\arraycolsep{1pt}
{}_2 \phi_1\left(\begin{matrix}
-\frac{q}{y}& &A& \\&y &\end{matrix}; q^2, y\right)
=\lim_{y\rightarrow0}\,\frac{(A;q^2)_{\infty}(-q;q^2)_{\infty}}{(y;q^2)_{\infty}^2}\cdot 
\setlength\arraycolsep{1pt}
{}_2 \phi_1\left(\begin{matrix}
 \frac{y}{A}& &y& \\&-q &\end{matrix}; q^2, A\right) \\
&\,\,=(A;q^2)_{\infty}\, (-q;q^2)_{\infty} \sum_{n\geq0}\frac{A^n}{(q^2;q^2)_n(-q;q^2)_n}
=(A;q^2)_{\infty}\, (-q;q^2)_{\infty} \sum_{n\geq0}\frac{A^n}{(-q;-q)_{2n}}.
\end{align*}
Hence, we obtain
\begin{align*} 
F(A;-q)&=(A;q^2)_{\infty}(q;q^2)_{\infty} \cdot \frac12\sum_{n\geq0}\frac{A^{n/2}(1+(-1)^n)}{(q;q)_{n}} \\
&=\frac12\cdot (A;q^2)_{\infty}(q;q^2)_{\infty}\cdot \left(\frac1{(\sqrt{A};q)_{\infty}}+\frac1{(-\sqrt{A};q)_{\infty}}\right) \\
&=\frac12\cdot (q;q^2)_{\infty}\left[ (-\sqrt{A};q)_{\infty}+(\sqrt{A};q)_{\infty}\right].
\end{align*}

\noindent
(2) By applying Heine's transformation to $G(A;q)$, we obtain 
\begin{align*}
G(A;q)&:=\lim_{y\rightarrow0}\, \setlength\arraycolsep{1pt}
{}_2 \phi_1\left(\begin{matrix}
- \frac{q}{y}& &A& \\&y &\end{matrix}; q^2, yq^2\right)
=\lim_{y\rightarrow0}\,\frac{(A;q^2)_{\infty}(-q^3;q^2)_{\infty}}{(y;q^2)_{\infty}(y q^2;q^2)_{\infty}}\cdot 
\setlength\arraycolsep{1pt}
{}_2 \phi_1\left(\begin{matrix}
\frac{y}{A}& &yq^2& \\&-q^3 &\end{matrix}; q^2, A\right)  \\
&\,\,=(A;q^2)_{\infty}\, (-q^3;q^2)_{\infty} \sum_{n\geq0}\frac{A^n}{(q^2;q^2)_n(-q^3;q^2)_n}
=(A;q^2)_{\infty}\, (-q;q^2)_{\infty} \sum_{n\geq0}\frac{A^n}{(-q;-q)_{2n+1}} \\
&=\frac12\cdot (A;q^2)_{\infty}(-q;q^2)_{\infty}\sum_{n\geq0}\frac{A^{(n-1)/2}(1-(-1)^n)}{(-q;-q)_n}.
\end{align*}
The same argument that completed (1) proves part (2). 
\end{proof}

\begin{proof}[Proof of Theorem~\ref{newconj}]
For self-conjugate partitions, we have that $n_1(\lambda)$ is odd if and only if $\lambda$ is of Type 1. 
This follows as 1-hooks generally come in pairs due to self-conjugation. The only counterexamples are the lower right entries of the Durfee squares of Type 1 self-conjugate partitions, which are fixed points under conjugation. Furthermore, the only 1-hooks that appear in a twin partition, say the cells to the right of the Durfee square, appear in the lower right hand entry of each rectangular block, which might consist of multiple repeated parts (see Figure~\ref{AjitsPictures}).
Therefore, \eqref{D1} and \eqref{D2} are represented as
\begin{displaymath}
\begin{split}
D_1(x;q)&=\sum_{n\geq 1}xq^{n^2}\prod_{j=1}^{n-1}(1+x^2q^{2j}+x^2q^{4j}+\cdots),\\
D_2(x;q)&=\sum_{n\geq 1}q^{n^2}(x^2q^{2n}+x^2q^{4n}+\cdots)\prod_{j=1}^{n-1}(1+x^2q^{2j}+x^2q^{4j}+\cdots).
\end{split}
\end{displaymath}
By direct algebraic manipulation, we obtain
\begin{equation}\label{gen1}
\sum_{\lambda\in \mathcal{SC}}x^{n_1(\lambda)}q^{|\lambda|}
=1+x\cdot \sum_{n\geq 1}\frac{q^{n^2}\cdot (1-(1-x)q^{2n})\cdot ((1-x^2)q^2;q^2)_{n-1}}{(q^2;q^2)_{n}}.
\end{equation}

\noindent
By \eqref{gen1} and the definition of $F$ and $G$, we obtain 
\begin{displaymath}
\begin{split}
\sum_{\lambda\in \mathcal{SC}}x^{n_1(\lambda)}q^{|\lambda|}
&=1+\frac{1}{x}\cdot \sum_{n\geq 1}\frac{q^{n^2}\cdot (1-(1-x)q^{2n})\cdot (1-x^2;q^2)_{n}}{(q^2;q^2)_{n}}\nonumber\\
&=1+\frac{1}{x}\cdot \left[(F(1-x^2;q)-1)+(x-1)\cdot (G(1-x^2;q)-1)\right]. \label{Thm3}
\end{split}
\end{displaymath}

\noindent
To prove (1) and (2), we make use of the infinite product identities in Lemma~\ref{George_Lm} for $F(x^2-1;q)$ and $G(x^2-1;q).$
This substitution immediately proves (1).
To prove (2), we take one further step and utilize Euler's summation formula for infinite products (for example, see 
\cite[(2.5)]{Andrews_Euler})
 $$
 (-z;q)_{\infty}=\sum_{n\geq0}\frac{q^{\binom{n+1}2}z^n}{(q;q)_n}. $$
\end{proof}

\subsection{Proof of Theorem~\ref{vital}}
We require the generating functions of certain arithmetic functions that enumerate self-conjugate partitions. It is an exercise to see that
\begin{equation}\label{scn}
\sum_{n\geq0}sc(n)\, q^n= (-q;q^2)_{\infty},
\end{equation}
where $sc(n)$ is the number of self-conjugate partitions of size $n.$ We also require $asc_t(n),$ the number of $t$-core partitions that are self-conjugate.
The following generating functions were derived by Garvan, Kim, and Stanton.

\begin{lemma}\cite[(7.1a) and (7.1b)]{GKS} \label{generatingSCcores} The following are true.

\noindent
(1) If $t$ is even, then we have 
$$
\sum_{n\geq0} asc_t(n)\, q^n= (-q;q^2)_{\infty}\cdot (q^{2t};q^{2t})_{\infty}^{\frac{t}{2}}. 
$$

\noindent
(2) If $t$ is odd, then we have 
$$
\sum_{n\geq0} asc_t(n)\, q^n=(-q;q^2)_{\infty}\cdot
 \frac{(q^{2t};q^{2t})_{\infty}^{\frac{t-1}{2}}}{(-q^t;q^{2t})_{\infty}}.
$$
\end{lemma}

\begin{proof}[Proof of Theorem~\ref{vital}]

\noindent
(1)  Given $\lambda\in\mathcal{SC}$, Theorem~\ref{LD} associates the unique pair $(\omega;\underline{\nu})$ where $\text{core}_t(\lambda):=\omega\in\mathcal{SC}_t$. Let $\vert\underline{\nu}\vert:=\sum_{i=0}^{t-1}\vert\nu^{(i)}\vert$ and $n_1(\underline{\nu}):=\sum_{i=0}^{t-1}n_1(\nu^{(i)}).$ Thanks to Theorem~\ref{LD} (SC4), we have that
	$$n_t(\lambda)=t\sum_{i=0}^{t-1} n_1(\nu^{(i)}) .$$
	By combining this fact with Theorem~\ref{LD} (SC3), we find that
	\begin{align*}
	\sum_{\substack{\lambda\in\mathcal{SC} \\ \text{core}_t(\lambda)=\omega}}
	x^{n_t(\lambda)}\,q^{\vert\lambda\vert}
	&=q^{\vert\omega\vert}\sum_{\substack{\lambda\in\mathcal{SC} \\ \text{core}_t(\lambda)=\omega}}
	x^{n_t(\lambda)}\,q^{\vert\lambda\vert-\vert\omega\vert} 
	=q^{\vert\omega\vert} \sum_{\underline{\nu}\in\mathcal{P}^t}
	x^{ n_1(\underline{\nu})}\,q^{t\vert\underline{\nu}\vert} \, .
	\end{align*}
	Theorem~\ref{LD} (SC2) justifies that the $t$-quotient $\underline{\nu}$ is uniquely determined once we know its first $t/2$ components (be mindful that these can attain \emph{any} partition) and $\vert\nu^{(i)}\vert=\vert\nu^{(t-i-i)}\vert$ as well as $\mathcal{H}(\nu^{(i)})=\mathcal{H}(\nu^{(t-1-i)})$ for $0\leq i\leq t/2-1$
Therefore, we can break up the inner sum as follows:
	\begin{align*}
	\sum_{\omega\in\mathcal{SC}_t} q^{\vert\omega\vert}
	\sum_{\substack{\lambda\in\mathcal{SC} \\ \text{core}_t(\lambda)=\omega}}
	x^{n_t(\lambda)} \, q^{\vert\lambda\vert-\vert\omega\vert}
	&=\sum_{\omega\in\mathcal{SC}_t}q^{\vert\omega\vert} \cdot
	\left(\sum_{\nu\in\mathcal{P}} x^{2n_1(\nu)} \, q^{2t\vert\nu\vert} \right)^{\frac{t}{2}}.
       \end{align*}
Han's formula \eqref{Han-Thm1.4} then implies that
\begin{align*}
\sum_{\omega\in\mathcal{SC}_t}q^{\vert\omega\vert} \cdot
	\left(\sum_{\nu\in\mathcal{P}} x^{2n_1(\nu)} \, q^{2t\vert\nu\vert} \right)^{\frac{t}{2}}
	&=\sum_{\omega\in\mathcal{SC}_t}q^{\vert\omega\vert} \cdot
\frac{((1-x^2)q^{2t};q^{2t})_{\infty}^{\frac{t}{2}}}{(q^{2t};q^{2t})_{\infty}^{\frac{t}{2}}}.
\end{align*}
For the outer sum, one inserts Lemma~\ref{generatingSCcores}~(1) to obtain
\begin{align*}
     \sum_{\omega\in\mathcal{SC}_t}q^{\vert\omega\vert} \cdot
\frac{((1-x^2)q^{2t};q^{2t})_{\infty}^{\frac{t}{2}}}{(q^{2t};q^{2t})_{\infty}^{\frac{t}{2}}}
	&=(-q;q^2)_{\infty}\cdot (q^{2t};q^{2t})_{\infty}^{\frac{t}{2}}  \cdot 
\frac{((1-x^2)q^{2t};q^{2t})_{\infty}^{\frac{t}{2}}}{(q^{2t};q^{2t})_{\infty}^{\frac{t}{2}}}\\
&= (-q;q^2)_{\infty} \cdot ((1-x^2)q^{2t};q^{2t})_{\infty}^{\frac{t}{2}}.
	\end{align*}
	This completes the proof of (1).

\medskip
\noindent
(2) The proof is analogous to the proof of (1).  However, one must take into account  Theorem~\ref{LD} (SC2$'$) and (SC3).  In this case, the Littlewood bijection generates an extra self-conjugate partition $\mu$. Hence, we have
\begin{align*}
\sum_{\lambda\in\mathcal{SC}} x^{n_t(\lambda)}\, q^{|\lambda|}
&=\sum_{\omega\in\mathcal{SC}_t} q^{\vert\omega\vert} \sum_{\underline{\nu}\in\mathcal{P}^t}
	x^{n_1(\underline{\nu})}\,q^{t\,\vert\underline{\nu}\vert} 
=\sum_{\omega\in\mathcal{SC}_t}q^{\vert\omega\vert}
	\left(\sum_{\nu\in\mathcal{P}} x^{2n_1(\nu)} \, q^{2t\vert\nu\vert} \right)^{\frac{t-1}{2}}
\sum_{\mu\in\mathcal{SC}} x^{n_1(\mu)}\, q^{t\vert\mu\vert}.
\end{align*}
Thanks to Han's formula (\ref{Han-Thm1.4}) and Lemma~\ref{generatingSCcores}~(2),  we get
\begin{align*}
\sum_{\omega\in\mathcal{SC}_t}q^{\vert\omega\vert}
	\left(\sum_{\nu\in\mathcal{P}} x^{2n_1(\nu)} \, q^{2t\vert\nu\vert} \right)^{\frac{t-1}{2}}
&=(-q;q^2)_{\infty}\cdot \frac{((1-x^2)q^{2t};q^{2t})_{\infty}^{\frac{t-1}{2}}}{(-q^t;q^{2t})_{\infty}}
:=\frac{A_t(x;q^t)}{(-q^t;q^{2t})_{\infty}}.
\end{align*}
Thanks to Theorem~\ref{newconj} (2) and the definition of $H^{\star}(x;q)$ in (\ref{HstarFormula}), after replacing $q$ by $q^t$ we obtain
$$\sum_{\lambda\in\mathcal{SC}} x^{n_t(\lambda)}\, q^{\lambda}=A_t(x;q^t)\cdot H^{\star}(x;q^t),$$
This proves (2).
\end{proof}

\subsection{Proof of Theorem~\ref{ExactFormula}}
To prove (1), we differentiate the identity in Theorem \ref{vital} (1) with respect to $x.$  After letting $x=1,$ one obtains
$$
\sum_{n\geq1} a_t^{\star}(n)q^n=t \cdot \frac{q^{2t}\cdot (-q;q^2)_{\infty}}{1-q^{2t}}=
	t\cdot (q^{2t}+q^{4t}+q^{6t}+\cdots)\cdot (-q;q^2)_{\infty}.
$$
The second claim follows from (\ref{scn}) and the fact that the generating function for $q^{\star}(n)$ is 
$$
\sum_{n\geq0}q^{\star}(n)q^n=\prod_{n=0}^{\infty}(1+q^{2n+1})=(-q;q^2)_{\infty}.
$$
These two formulas immediately give the desired result.

\smallskip
The proof of (2) is similar to (1) and follows by differentiating the formula in Theorem~\ref{vital} (2).

\end{document}